%% file: m.tex
\documentclass[reqno,12pt]{amsart}

\include{environ}
\include{symbols}

\def\calg#1{{\mathcal #1}}
\def\cal{\mathcal}
\def\size{\mathrm{diam}}
\def\supp{\mathrm{supp}}
\newcommand{\Na}{{\bb N}}

\begin{document}

\title[Optimality of multilevel preconditioners for local 3D refinement]
      {Optimality of multilevel preconditioners\\
       for local mesh refinement in three dimensions}

\author[B. Aksoylu]{Burak Aksoylu}
\email{burak@cct.lsu.edu}
\address{Department of Mathematics,
         Louisiana State University,
         Baton Rouge, LA 70803, USA}
\address{Center for Computation and Technology,
         Louisiana State University,
         Baton Rouge, LA 70803, USA}
\thanks{The first author was supported in part by the Burroughs
Wellcome Fund, in part by NSF (ACI-9721349, DMS-9872890), and in part
by DOE (W-7405-ENG-48/B341492). Other support was provided by Intel,
Microsoft, Alias$|$Wavefront, Pixar, and the Packard Foundation.}

\author[M. Holst]{Michael Holst}
\email{mholst@math.ucsd.edu}
\address{Department of Mathematics\\
         University of California at San Diego\\ 
         La Jolla, CA 92093, USA}
\thanks{The second author was supported in part by NSF 
(CAREER Award~DMS-9875856 and standard grants DMS-0208449, DMS-9973276, 
DMS-0112413), in part by DOE (SCI-DAC-21-6993), and in part by a 
Hellman Fellowship.}

\date{December 12, 2005}


\input{abs}

\maketitle


\vspace*{-1.2cm}
{\footnotesize
\tableofcontents
}

\input{body}

\input{ack}

\bibliographystyle{siam}
\bibliography{m}


\vspace*{0.5cm}

\end{document}

%% file: environ.tex
\pdfoutput=1

\usepackage{color} 
\usepackage{ifpdf}
\ifpdf
    \usepackage[pdftex]{graphicx}
    \usepackage[pdftex]{hyperref}
    \hypersetup{
        unicode=false,          
        pdftoolbar=true,        
        pdfmenubar=true,        
        pdffitwindow=false,     
        pdfstartview={FitH},    
        pdftitle={MCP Article},      
        pdfauthor={Michael Holst},   
        pdfsubject={Mathematics},    
        pdfcreator={Michael Holst},  
        pdfproducer={Michael Holst}, 
        pdfkeywords={PDE, analysis, mathematical physics}, 
        pdfnewwindow=true,      
        colorlinks=true,        
        linkcolor=red,          
        citecolor=blue,         
        filecolor=magenta,      
        urlcolor=cyan           
    }

\else
    \usepackage{graphicx}
    \usepackage{pstricks}

\fi

\usepackage{times}
\usepackage{amsfonts}

\usepackage{amsmath}
\usepackage{amsthm}
\usepackage{amssymb}
\usepackage{amsbsy}
\usepackage{amscd}

\usepackage{enumerate}
\usepackage{verbatim}
\usepackage{subfigure}

\newtheorem{theorem}{Theorem}[section]
\newtheorem{corollary}[theorem]{Corollary}
\newtheorem{lemma}[theorem]{Lemma}

\newtheorem{assumption}[theorem]{Assumption}

\newtheorem{remark}[theorem]{Remark}

\numberwithin{equation}{section}  

  \newcounter{mnote}
  \let\oldmarginpar\marginpar
    \renewcommand\marginpar[1]{\-\oldmarginpar[\raggedleft\footnotesize #1]%
    {\raggedright\footnotesize #1}}

\definecolor{myblue}{rgb}{0.2,0.2,0.7}
\definecolor{mygreen}{rgb}{0,0.6,0}
\definecolor{mycyan}{rgb}{0,0.6,0.6}
\definecolor{myred}{rgb}{0.9,0.2,0.2}
\definecolor{mymagenta}{rgb}{0.9,0.2,0.9}
\definecolor{mywhite}{rgb}{1.0,1.0,1.0}
\definecolor{myblack}{rgb}{0.0,0.0,0.0}

%% file: symbols.tex
\newcommand{\beq}{\begin{equation}}
\newcommand{\eeq}{\end{equation}}
\newcommand{\beqa}{\begin{eqnarray}}
\newcommand{\eeqa}{\end{eqnarray}}
\newcommand{\bb}{{\bf b}}



%% file: abs.tex
\begin{abstract}
In this article, we establish optimality of the 
Bramble-Pasciak-Xu (BPX) norm equivalence and optimality of the
wavelet modified (or {\em stabilized}) hierarchical basis (WHB)
preconditioner in the setting of local 3D mesh refinement.
In the analysis of WHB methods, a critical first step
is to establish the optimality of BPX norm equivalence for the
refinement procedures under consideration.  While the available
optimality results for the BPX norm have been constructed primarily in
the setting of uniformly refined meshes, a notable exception is the
local 2D red-green result due to Dahmen and Kunoth.  The purpose of
this article is to extend this original 2D optimality result to the
local 3D red-green refinement procedure introduced by
Bornemann-Erdmann-Kornhuber (BEK), and then to use this result to
extend the WHB optimality results from the quasiuniform setting to
local 2D and 3D red-green refinement scenarios.
The BPX extension is reduced to establishing that locally enriched
finite element subspaces allow for the construction of a scaled basis
which is formally Riesz stable.  This construction turns out to rest
not only on shape regularity of the refined elements, but also
critically on a number of geometrical properties we establish between
neighboring simplices produced by the BEK refinement procedure.
It is possible to show that the number of degrees of freedom used for
smoothing is bounded by a constant times the number of degrees of freedom
introduced at that level of refinement, indicating that a practical
implementable version of the resulting BPX preconditioner for the BEK
refinement setting has provably optimal (linear) computational
complexity per iteration.  An interesting implication of the
optimality of the WHB preconditioner is the {\em a priori}
$H^1$-stability of the $L_2$-projection.  The existing 
{\em a posteriori} approaches in the literature dictate a 
reconstruction of the mesh if such conditions cannot be satisfied.
The theoretical framework employed supports arbitrary spatial 
dimension $d \geq 1$ and requires no coefficient smoothness 
assumptions beyond those required for well-posedness in $H^1$.
\end{abstract}

%% file: body.tex

\section{Introduction}

In this article, we analyze the impact of local mesh refinement on the
stability of multilevel finite element spaces and on optimality
(linear space and time complexity) of multilevel preconditioners.
Adaptive refinement techniques have become a crucial tool for many
applications, and access to optimal or near-optimal multilevel
preconditioners for locally refined mesh situations is of primary
concern to computational scientists.  The preconditioners which can be
expected to have somewhat favorable space and time complexity in such
local refinement scenarios are 
the hierarchical basis (HB) method~\cite{BaDuYs88},
the Bramble-Pasciak-Xu (BPX) preconditioner~\cite{BPX90},
and the wavelet modified (or stabilized)
hierarchical basis (WHB) method~\cite{VaWa1}.
While there are
optimality results for both the BPX and WHB preconditioners in the
literature, these are primarily for quasiuniform meshes and/or two
space dimensions (with some exceptions noted below).  In particular,
there are few hard results in the literature on the optimality of
these methods for various realistic local mesh refinement hierarchies,
especially in three space dimensions. In this article, the first in a
series of two articles~\cite{AkBoHo03} on local refinement and
multilevel preconditioners, we first assemble optimality results for
the BPX norm equivalence in local refinement scenarios in three spacial
dimensions.  Building on the extended BPX results, we then develop
optimality results for the WHB method in local refinement
settings. The material forming this series is based on the first
author's Ph.D. dissertation~\cite{Ak01phd} and comprehensive presentation
of this article can be found 
in~\cite{AkBoHo04-techReport,AkHo05-techReportI,AkHo05-techReportII,AkKhSc03}.

Through some topological or geometrical abstraction, if local
refinement is extended to $d$ spatial dimensions, then the main
results are valid for any dimension $d \geq 1$ and for nonsmooth PDE
coefficients $p \in L_{\infty}(\Omega)$. Throughout this article, we
consider primarily the $d=3$ case. But, when the abstraction to
generic $d$ is clear, we simply state the argument by using this
generic $d$.

The problem class we focus on here is linear second order partial
differential equations (PDE) of the form:
\begin{equation}  \label{modelProb}
- \nabla \cdot (p ~ \nabla u) + q ~ u = f,
~~~u = 0 ~~\text{on}~ \partial \Omega.
\end{equation}
Here, $f \in L_2(\Omega)$, $p,q \in L_\infty(\Omega)$, $p:\Omega
\rightarrow L(\Re^d, \Re^d)$, $q:\Omega \rightarrow \Re$, where $p$ is
a symmetric positive definite matrix function, and where $q$ is a
nonnegative function. Let ${\cal T}_0$ be a shape regular and
quasiuniform initial partition of $\Omega$ into a finite number of $d$
simplices, and generate ${\cal T}_1, {\cal T}_2, \ldots$ by refining
the initial partition using red-green local refinement strategies in
$d=3$ spatial dimensions.  Denote as ${\cal S}_j$ the simplicial
linear $C^0$ finite element space corresponding to ${\cal T}_j$
equipped with zero boundary values. The set of nodal basis functions
for ${\cal S}_j$ is denoted by $\Phi^{(j)} = \{ \phi_i^{(j)}
\}_{i=1}^{N_j}$ where $N_j = \mbox{dim}~ {\cal S}_j$ is equal to the
number of interior nodes in ${\cal T}_j$, representing the number of
degrees of freedom in the discrete space.  Successively refined finite
element spaces will form the following nested sequence:
$$
{\cal S}_0 \subset {\cal S}_1 \subset \ldots \subset {\cal S}_j \subset \ldots
\subset H_0^1(\Omega).
$$

Let the bilinear form and the functional associated with
the weak formulation of~(\ref{modelProb}) be denoted as
$$
a(u,v) = \int_{\Omega} p~ \nabla u \cdot \nabla v + q~u~v ~dx,
~~~b(v) = \int_{\Omega} f~v~dx,~~~u,v \in H_0^1(\Omega).
$$
We consider primarily the following Galerkin formulation: Find
$u \in {\cal S}_j$, such that
\begin{equation} \label{FEproblem}
a(u,v) = b(v),~~~\forall v \in {\cal S}_j.
\end{equation}
The finite element approximation in ${\cal S}_j$ has the form
$
u^{(j)} = \sum_{i=1}^{N_j} u_i \phi_i^{(j)},
$
where $u=(u_1, \ldots, u_{N_j})^T$ denotes the coefficients of $u^{(j)} $ with
respect to $\Phi^{(j)}$.
The resulting {\em discretization operator} 
$A^{(j)}=\{a(\phi_k^{(j)},\phi_l^{(j)})\}_{k,l=1}^{N_j}$
must be inverted numerically to determine the coefficients 
$u$ from the linear system:
\begin{equation} \label{MATproblem}
A^{(j)} u = F^{(j)},
\end{equation}
where $F^{(j)} =\{b(\phi_l^{(j)})\}_{l=1}^{N_j}$.
Our task is to solve~(\ref{MATproblem}) with optimal (linear) complexity
in both storage and computation, where the finite element spaces
${\cal S_j}$ are built on locally refined meshes.

Optimality of the BPX norm equivalence with generic local refinement was
shown by Bramble and Pasciak~\cite{BrPa93}, where the impact of the
local smoother and the local projection operator on the estimates was
carefully analyzed.  The two primary results on optimality of the BPX
norm equivalence in the local refinement settings are due to Dahmen and
Kunoth~\cite{DaKu92} and Bornemann and Yserentant~\cite{BoYs93}. Both
works consider only two space dimensions, and in particular, the
refinement strategies analyzed are restricted 2D red-green refinement
and 2D red refinement, respectively.  In this paper, we extend the
framework developed in~\cite{DaKu92} to a practical, implementable 3D
local red-green refinement procedure introduced by
Bornemann-Erdmann-Kornhuber (BEK)~\cite{BoErKo}.  We will refer to
this as the BEK refinement procedure.

HB methods~\cite{BaDuYs88,BaActa96,Ys86}
are particularly attractive in the local refinement setting
because (by construction) each iteration has linear (optimal)
computational and storage complexity.  Unfortunately, the resulting
preconditioner is not optimal due to condition number growth: in two
dimensions the growth is slow, and the method is quite effective
(nearly optimal), but in three dimensions the condition number grows
much more rapidly with the number of unknowns~\cite{On89}.
To address this
instability, one can employ $L_2$-orthonormal wavelets in place of the
hierarchical basis giving rise to an optimal preconditioner~\cite{Ja92}.
However, the complicated nature of traditional wavelet bases, in
particular the non-local support of the basis functions and
problematic treatment of boundary conditions, severely limits
computational feasibility.  WHB methods have been
developed~\cite{VaWa2,VaWa1} as an alternative, and they can be
interpreted as a wavelet modification (or {\em stabilization}) of the
hierarchical basis.  These methods have been shown to optimally
stabilize the condition number of the systems arising from
hierarchical basis methods on quasiuniform meshes in both two and
three space dimensions, and retain a comparable cost per iteration.

There are two main results and one side result in this article.  The
main results establish the optimality of the BPX norm equivalence and
also optimality of the WHB preconditoner---as well as optimal
computational complexity per iteration---for the resulting locally
refined 3D finite element hierarchy.  Both the BPX and WHB
preconditioners under consideration are additive Schwarz
preconditioners.  The BPX analysis here heavily relies on the
techniques of the Dahmen-Kunoth~\cite{DaKu92} framework and can be
seen as an extension to three spatial dimensions with the realistic
BEK refinement procedure~\cite{BoErKo} being the application of
interest.  The WHB framework relies on the optimality of the BPX norm
equivalence. Hence, the WHB results are established after the BPX
results.

The side result is the $H^1$-stability of $L_2$-projection onto finite
element spaces built through the BEK local refinement procedure.  This
question is currently under intensive study in the finite element
community due to its relationship to multilevel preconditioning. The
existing theoretical results, due primarily to Carstensen~\cite{Ca00}
and Bramble-Pasciak-Steinbach~\cite{BrPaSt00} involve {\em a
posteriori} verification of somewhat complicated mesh conditions after
local refinement has taken place.  If such mesh conditions are not
satisfied, one has to redefine the mesh.  However, an interesting
consequence of the BPX optimality results for locally refined 2D and
3D meshes established here is $H^1$-stability of $L_2$-projection
restricted to the same locally enriched finite element spaces. This
result appears to be the first {\em a priori} $H^1$-stability result
for $L_2$-projection on finite element spaces produced by practical
and easily implementable 2D and 3D local refinement procedures.

{\bf\em Outline of the paper.}  
In \S\ref{sec:normEquivalence}, we introduce some basic
approximation theory tools used in the analysis such as Besov spaces
and Bernstein inequalities. The framework for the main norm
equivalence is also established here. In \S\ref{sec:3DRed-Green}, we
list the BEK refinement conditions. We give several theorems about the
generation and size relations of the neighboring simplices, thereby
establishing local (patchwise) quasiuniformity. This gives rise to an
$L_2$-stable Riesz basis in~\S\ref{sec:stableRieszBasis}; one can then
establish the Bernstein inequality.
In \S\ref{sec:localComplexity}, we explicitly give an upper bound for
the nodes introduced in the refinement region. This implies that one
application of the BPX preconditioner to a function has linear
(optimal) computational complexity.  In \S\ref{sec:optBPX}, we use the
geometrical results from \S\ref{sec:3DRed-Green} to extend the 2D
Dahmen-Kunoth results to the 3D BEK refinement procedure by
establishing the desired norm equivalence.  While it is not possible
to establish a Jackson inequality due to the nature of local
adaptivity, in~\S\ref{sec:reason} the remaining inequality in the norm
equivalence is handled directly using approximation theory tools, as
in the original work~\cite{DaKu92}. In~\S\ref{sec:WHBprecond}, we
introduce the WHB preconditioner as well as the operator used in its
definition. In~\S\ref{sec:assump-stabp}, we state the fundamental
assumption for establishing basis stability and set up the main
theoretical results for the WHB framework, namely, optimality of
the WHB preconditioner in the 2D and 3D local red-green
refinements. The results in~\S\ref{sec:assump-stabp} rest completely
on the BPX results in~\S\ref{sec:optBPX} and on the Bernstein
inequalities, the latter of which rest on the geometrical results
established in~\S\ref{sec:3DRed-Green}. The first {\em a priori}
$H^1$-stability result for $L_2$-projection on the finite element
spaces produced is established in~\S\ref{sec:H1StableL2}. We conclude
in~\S\ref{sec:conclusion}.

\section{Preliminaries and the main norm equivalence}
\label{sec:normEquivalence}
The basic restriction on the refinement procedure is that it remains
\emph{nested}. In other words, tetrahedra of level $j$ which are not
candidates for further refinement will never be touched in the future.
Let $\Omega_j$ denote the refinement region, namely, the union of the
supports of basis functions which are introduced at level $j$.  Due to
nested refinement $\Omega_j \subset \Omega_{j-1}$. Then the following
hierarchy holds:
\begin{equation} \label{eqn:omega}
\Omega_J \subset \Omega_{J-1} \subset \cdots \subset \Omega_0 = \Omega.
\end{equation}

In the local refinement setting, in order to maintain optimal computational 
complexity, the smoother is restricted to a local space
$\tilde{{\cal S}}_j$, typically 
\begin{equation} \label{subset:localSpace}
{\cal S}_j^f \subseteq \tilde{{\cal S}}_j \subset {\cal S}_j,
\end{equation}
where ${\cal S}_j^f := (I_j - I_{j-1})~{\cal S}_{j}$ and 
$I_j:L_2(\Omega) \rightarrow {\cal S}_j$ denotes the finite
element interpolation operator. 
Degrees of freedom (DOF) corresponding to ${\cal S}_j^f$
and $\tilde{{\cal S}}_j$ will be denoted by ${\cal N}_j^f$ and 
$\tilde{\calg{N}}_j$ respectively where $f$ stands for \emph{fine}. 
(\ref{subset:localSpace}) indicates that
${\cal N}_j^f \subseteq \tilde{\calg{N}}_j$,
typically, $\tilde{\calg{N}}_j$ consists of fine DOF and their corresponding 
coarse fathers.

The BPX preconditioner (also known as parallelized or additive
multigrid) is defined as follows:
\begin{equation} \label{id:BPX}
X u := \sum_{j=0}^J 2^{j(d-2)} \sum_{i \in \tilde{\calg{N}}_j}
(u,\phi_i^{(j)}) \phi_i^{(j)}.
\end{equation}
Success of the BPX preconditioner in locally refined regimes 
relies on the fact the BPX smoother acts on a local space as in
(\ref{subset:localSpace}). As mentioned above, it acts on a slightly
bigger set than fine DOF (examples of these are given in~\cite{BrPa92}).
Choice of such a set is crucial because computational cost per
iteration will eventually determine the overall computational complexity
of the method. Hence in~\S\ref{sec:localComplexity}, we show that
the overall computational cost of the smoother is $O(N)$, meaning that 
the BPX preconditioner is optimal per iteration.  We would like to
emphasize that one of the the main goals of this paper, as in the earlier
works of Dahmen-Kunoth~\cite{DaKu92} and Bornemann-Yserentant~\cite{BoYs93}
in the purely two-dimensional case, is to establish the
optimality of the BPX norm equivalence:
\begin{equation} \label{ineq:BPXnormEquiv}
c_1 \sum_{j=0}^J 2^{2j} \|(Q_{j} - Q_{j-1})u\|_{L_2}^2 \leq 
\|u\|_{H^1}^2 \leq 
c_2 \sum_{j=0}^J 2^{2j} \|(Q_{j} - Q_{j-1})u\|_{L_2}^2,
\end{equation}
where $Q_j$ is the $L_2$-projection. 
We note that in the uniform refinement setting, it is straight-forward
to link the BPX norm equivalence to the optimality of the BPX
preconditioner:
$$
c_1 (Xu,u) \leq \|u\|_{H^1}^2 \leq c_2 (Xu,u),
$$
due to the projector relationships between the $Q_j$ operators.
However, in the local refinement scenario the precise link between
the norm equivalence and the preconditioner is more subtle and remains
essentially open.

The rest of this section is dedicated to setting up the framework to 
establish the main norm equivalence (\ref{ineq:BPXnormEquiv}) which 
will be formalized in Theorem~\ref{thm:main} at the end of this section.
We borrow several tools from approximation theory, including the
modulus of smoothness, $\omega_k (f,t,\Omega)_p$, which is a finer scale
of smoothness than differentiability. It is a central tool in the analysis
here and it naturally gives rise to the notion of {\em Besov spaces}. For 
further details and definitions, see~\cite{DaKu92,Os94book}.  Besov 
spaces are defined to be the collection of functions $f \in L_p(\Omega)$
with a finite Besov norm defined as follows:
$$
\|f\|_{B_{p,q}^s(\Omega)}^q :=
\|f\|_{L_p(\Omega)}^q + |f|_{B_{p,q}^s(\Omega)}^q,
$$
where the seminorm is given by
$$
|f|_{B_{p,q}^s(\Omega)} :=
\|\{2^{sj} \omega_k(f,2^{-j},\Omega)_p\}_{j \in \Na_0} \|_{l_q},
$$
with $k$ any fixed integer larger than $s$.

Besov spaces become the primary function space setting in the
analysis by realizing Sobolev spaces as Besov spaces:
$$
H^s(\Omega) \cong B_{2,2}^s(\Omega),~~~s>0.
$$
The primary motivation for employing the Besov space stems from the fact
that the characterization of functions which have a given upper bound for
the error of approximation sometimes calls for a finer scale of smoothness
that provided by Sobolev classes functions.

The Bernstein inequality is defined as:
\begin{equation} \label{ineq:Bernstein}
\omega_{k+1}(u,t)_p \leq c~(\min\{1, t2^J\})^\beta \|u\|_{L_p},~~~
u \in {\cal S}_j,~~j=0,\ldots,J,
\end{equation}
where $c$ is independent of $u$ and $j$. Usually $k=$ degree of the
element and in the case of linear finite elements $k=1$. Here $\beta$
is determined by the global smoothness of the approximation space as
well as $p$. For $C^r$ finite elements, $\beta =
\min\{1+r+\frac{1}{p}, k+1\}$.

Let $\theta_J$ be defined as follows.
\begin{equation} \label{def:lowerBound}
\theta_{j,J} := \sup_{u \in {\cal S}_J}
\frac{ \|u - Q_j u \|_{L_2} } { \omega_2 (u,2^{-j})_2 },~~~
\theta_J := \max \left\{ 1, \theta_{j,J}: j=0, \ldots, J \right\}.
\end{equation}
Following~\cite{DaKu92} we have then
\begin{theorem} \label{thm:main}
Suppose the Bernstein inequality (\ref{ineq:Bernstein}) holds for some
real number $\beta>1$. Then, for each $0<s<\min\{\beta, 2\}$, there
exist constants $0<c_1,~c_2 <\infty$ independent of $u \in {\cal
S}_J,~J=0,1,\ldots$, such that the following norm equivalence holds:
\begin{equation} \label{mainNormEquiv}
\frac{c_1}{ \theta_J^2 } \sum_{j=0}^J 2^{2j} \|(Q_{j} - Q_{j-1})u\|_{L_2}^2
\leq \|u\|_{H^1}^2  \leq c_2 
\sum_{j=0}^J 2^{2j} \|(Q_{j} - Q_{j-1})u\|_{L_2}^2,~~~u \in {\cal S}_J.
\end{equation}
\end{theorem}
\begin{proof}
See~\cite[Theorem 4.1]{DaKu92}.
\end{proof}

We would like to elaborate on the difficulities one faces within the
local refinement framework. In order Bernstein inequality to hold,
one needs to establish that the underlying basis is $L_2$-stable Riesz
basis as in (\ref{id:sBasisOswaldEquiv}). This crucial property
heavily depends on local quasiuniformity of the mesh. Hence, Bernstein
inequality is established in \S\ref{sec:optBPX} through local
quasiuniformity and $L_2$-stability of the basis in the Riesz sense.

A Jackson-type inequality cannot hold in a local refinement setting. This
poses a major difficulty in the analysis because one has to calculate
$\theta_J$ directly. The missing crucial piece of the optimal norm
equivalence in (\ref{mainNormEquiv}), namely, $\theta_J = O(1)$ as $J
\rightarrow \infty$, will be shown in (\ref{mainResult}) so that
(\ref{ineq:BPXnormEquiv}) holds. This required the operator
$\tilde{Q}_j$ to be bounded locally and to fix polynomials of degree 1
as will be shown in~\S\ref{sec:reason}.

\section{The BEK refinement procedure}
\label{sec:3DRed-Green}

Our interest is to show optimality of the BPX norm equivalence for the
local 3D red-green refinement introduced by
Bornemann-Erdmann-Kornhuber~\cite{BoErKo}.  This 3D red-green
refinement is practical, easy to implement, and numerical experiments
were presented in~\cite{BoErKo}. A similar refinement procedure was
analyzed by Bey~\cite{Be95}; in particular, the same green closure
strategy was used in both papers.  While these refinement procedures
are known to be asymptotically non-degenerate (and thus produce shape
regular simplices at every level of refinement), shape regularity is
insufficient to construct a stable Riesz basis for finite element
spaces on locally adapted meshes.  To construct a stable Riesz basis
we will need to establish patchwise quasiuniformity as
in~\cite{DaKu92}; as a result, $d$-vertex adjacency relationships that
are independent of shape regularity of the elements must be
established between neighboring tetrahedra as done in~\cite{DaKu92}
for triangles.

We first list a number of geometric assumptions we make concerning the
underlying mesh. Let $\Omega \subset \Re^3$ be a polyhedral domain.
We assume that the triangulation ${\cal T}_j$ of $\Omega$ at level $j$
is a collection of tetrahedra with mutually disjoint interiors which
cover $\Omega = \bigcup_{\tau \in {\cal T}_j } \tau$.  We want to
generate successive refinements ${\cal T}_0,{\cal T}_1, \ldots$ which
satisfy the following conditions:
\begin{assumption} \label{assump:rg1}
{\bf Nestedness:} Each tetrahedron (son) $\tau \in {\cal T}_j$ is
covered by exactly one tetrahedron (father) $\tau^\prime \in {\cal
T}_{j-1}$, and any corner of $\tau$ is either a corner or an edge
midpoint of $\tau^\prime$.
\end{assumption}
\begin{assumption} \label{assump:rg2}
{\bf Conformity:} The intersection of any two tetrahedra
$\tau,\tau^\prime \in {\cal T}_j$
is either empty, a common vertex, a common edge or a common face.
\end{assumption}
\begin{assumption} \label{assump:rg3}
{\bf Nondegeneracy:} The interior angles of all tetrahedra in the
refinement sequence ${\cal T}_0, {\cal T}_1, \ldots$ are bounded away
from zero.
\end{assumption}

\begin{figure}[htbp]
\centering{\includegraphics[width=150pt]{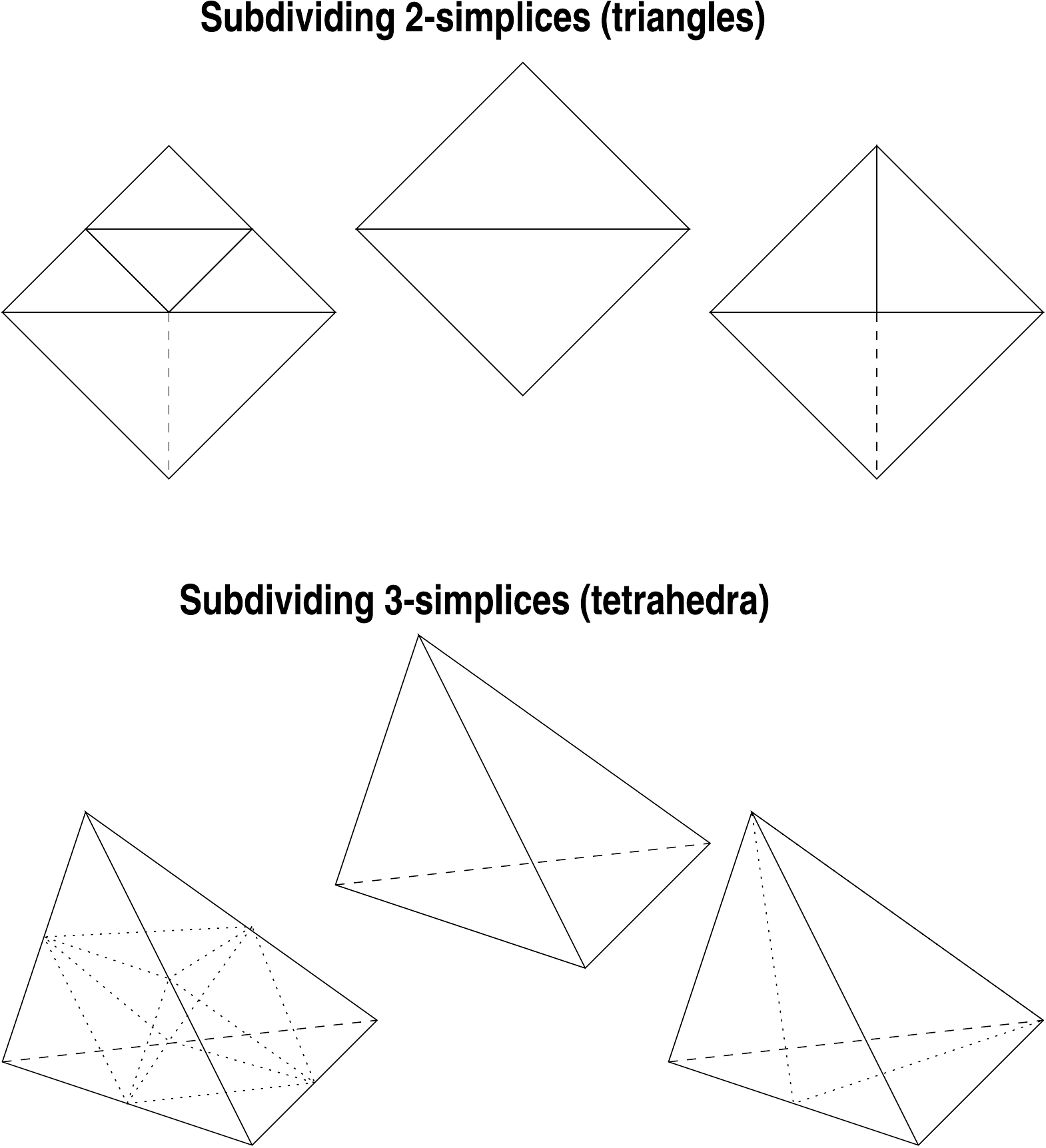}}
\caption{\label{fig:refinement}}
\end{figure}

A regular (red) refinement subdivides a tetrahedron $\tau$ into 8
equal volume subtetrahedra. We connect the edges of each face as in 2D
regular refinement.  We then cut off four subtetrahedra at the corners
which are congruent to $\tau$.  An octahedron with three
parallelograms remains in the interior. Cutting the octahedron along
the two faces of these parallelograms, we obtain four more
subtetrahedra which are not necessarily congruent to $\tau$. We choose
the diagonal of the parallelogram so that the successive refinements
always preserve nondegeneracy~\cite{Ak01phd,Be95,On94,Zh88}. A sketch
of regular refinement (octasection and quadrasection in 3D and 2D,
respectively) as well as bisection is given in Figure~\ref{fig:refinement}.

If a tetrahedron is marked for regular refinement, the resulting
triangulation violates conformity A.\ref{assump:rg2}. Nonconformity is
then remedied by irregular (green) refinement. In 3D, there are
altogether $2^6=64$ possible edge refinements, of which $62$ are
irregular.  One must pay extra attention to irregular refinement in
the implementation due to the large number of possible nonconforming
configurations.  Bey~\cite{Be95} gives a methodical way of handling
irregular cases. Using symmetry arguments, the $62$ irregular cases
can be divided into $9$ different types.  To ensure that the interior
angles remain bounded away from zero, we enforce the following
additional conditions.  (Identical assumptions were made
in~\cite{DaKu92} for their 2D refinement analogue.)

\begin{assumption} \label{assump:rg4}
Irregular tetrahedra are not refined further.
\end{assumption}
\begin{assumption} \label{assump:rg5}
Only tetrahedra $\tau \in {\cal T}_j$ with $L(\tau)=j$ are refined for
the construction of ${\cal T}_{j+1}$, where
$
L(\tau) = \min \left\{j:\tau \in {\cal T}_j \right\}
$
denotes the level of $\tau$.
\end{assumption}

One should note that the restrictive character of A.\ref{assump:rg4}
and A.\ref{assump:rg5} can be eliminated by a modification on the
sequence of the tetrahedralizations~\cite{Be95}.
On the other hand, it is straightforward to enforce both assumptions
in a typical local refinement algorithm by minor modifications of the
supporting datastructures for tetrahedral elements (cf.~\cite{Hols2001a}).
In any event, the proof technique (see (\ref{defn:dinosaurSet})
and (\ref{id:q-interFixes2})) requires both assumptions hold.
The last refinement condition enforced for the possible $62$ irregularly
refined tetrahedra is stated as the following.

\begin{assumption} \label{assump:rg6}
If three or more edges are refined and do not belong to a common face,
then the tetrahedron is refined regularly.
\end{assumption}

We note that the $d$-vertex adjacency generation bound for simplices
in $\Re^d$ which are adjacent at $d$ vertices {\em is the primary
result required in the support of a basis function so that
(\ref{id:stableBasisOswald}) holds, and depends delicately on the
particular details of the local refinement procedure rather than on
shape regularity of the elements.}  The generation bound for simplices
which are adjacent at $d-1,d-2, \ldots$ vertices follows by using the
shape regularity and the generation bound established for $d$-vertex
adjacency.  We provide rigorous generation bounds for all the
adjacency types mentioned in the lemmas to follow when $d=3$.  The 2D
version appeared in~\cite{DaKu92}; the 3D extension is as described
below.
\begin{lemma}
Let $\tau$ and $\tau^\prime$ be two tetrahedra in ${\cal T}_j$ sharing
a common face $f$. Then
\begin{equation} \label{ineq:face-adjacent}
|L(\tau)-L(\tau^\prime)| \leq 1.
\end{equation}
\end{lemma}
\begin{proof}
If $L(\tau)=L(\tau^\prime)$, then $0 \leq 1$, there is nothing to
show.  Without loss of generality, assume that $L(\tau) <
L(\tau^\prime)$. Proof requires a detailed and systematic analysis. To
show the line of reasoning, we first list the facts used in the proof:
\begin{enumerate}
\item $L(\tau^\prime) \leq j$ because by assumption $\tau^\prime \in
{\cal T}_j$.  Then, $L(\tau) < j$.
\item \label{item:never} By assumption $\tau \in {\cal T}_j$, meaning
that $\tau$ was never refined from the level it was born $L(\tau)$ to
level $j$.
\item \label{item:level} Let $\tau^{\prime \prime}$ be the father of
$\tau^\prime$. Then $L(\tau^{\prime \prime}) = L(\tau^\prime)-1 <j$.
\item $L(\tau) < L(\tau^\prime)$ by assumption, implying
$L(\tau) \leq L(\tau^{\prime \prime})$.
\item By (\ref{item:never}), $\tau$ belongs to all the triangulations from
$L(\tau)$ to $j$, in particular
$\tau \in {\cal T}_{L(\tau^{\prime \prime})}$, where
by (\ref{item:level}) $L(\tau^{\prime \prime})<j$.
\end{enumerate}

$f$ is the common face of $\tau$ and $\tau^\prime$ on level $j$. By
(5) both $\tau,~\tau^{\prime \prime} \in {\cal T}_{L(\tau^{\prime \prime})}$. 
Then, A.\ref{assump:rg2} implies that $f$ must still
be the common face of $\tau$ and $\tau^{\prime \prime}$. Hence,
$\tau^\prime$ must have been irregular.

On the other hand, $L(\tau) \leq L(\tau^\prime)-1 = L(\tau^{\prime
\prime})$.  Next, we proceed by eliminating the possibility that
$L(\tau)<L(\tau^{\prime \prime})$. If so, we repeat the above
reasoning, and $\tau^{\prime \prime}$ becomes irregular. $\tau^{\prime
\prime}$ is already the father of the irregular $\tau^\prime$,
contradicting A.\ref{assump:rg4} for level ${L(\tau^{\prime
\prime})}$.  Hence $L(\tau)=L(\tau^{\prime \prime})= L(\tau^\prime)-1$
concludes the proof.
\end{proof}

By A.\ref{assump:rg4} and A.\ref{assump:rg5}, every tetrahedron at any
${\cal T}_j$ is geometrically similar to some tetrahedron in ${\cal
T}_0$ or to a tetrahedron arising from an irregular refinement of some
tetrahedron in ${\cal T}_0$.  Then, there exist absolute constants
$c_1,~c_2$ such that
\begin{equation} \label{ineq:childFather1}
c_1~\size(\bar{\tau})~2^{-L(\tau)} \leq \size(\tau)
\leq c_2~\size(\bar{\tau})~2^{-L(\tau)},
\end{equation}
where $\bar{\tau}$ is the father of $\tau$ in the initial mesh. The
lemma below follows by shape regularity and
(\ref{ineq:face-adjacent}).

\begin{lemma} \label{lemma:edge-adjacent}
Let $\tau, \tau^\prime$ and $\zeta, \zeta^\prime$ be the tetrahedra in
${\cal T}_j$ sharing a common edge (two vertices) and a common vertex,
respectively. Then there exist finite numbers $V$ and $E$ depending on
the shape regularity such that
\begin{eqnarray}
|L(\tau)-L(\tau^\prime)| \leq V,  \label{ineq:vertex-adjacent} \\
|L(\zeta)-L(\zeta^\prime)| \leq E. \label{ineq:edge-adjacent}
\end{eqnarray}
\end{lemma}

Consequently, simplices in the support of a basis function are
comparable in size as indicated in (\ref{ineq:localRef1}). This is
usually called {\em patchwise quasiuniformity}. Furthermore, it was
shown in~\cite{Ak01phd} that patchwise quasiuniformity
(\ref{ineq:localRef1}) holds for 3D marked tetrahedron bisection by
Joe and Liu~\cite{JoLi95} and for 2D newest vertex bisection by
Sewell~\cite{Se72} and Mitchell~\cite{Mi88}. Due to the restrictive
nature of the proof technique (see (\ref{defn:dinosaurSet}) and
(\ref{id:q-interFixes2})), we focus on refinement procedures which
obey A.\ref{assump:rg4} and A.\ref{assump:rg5}.
However, due to the strong geometrical results available for purely
bisection-based local refinement procedures, it should be
possible to establish the main results of this paper for 
purely bisection-based strategies.
\begin{lemma}
There is a constant depending on the shape regularity of  ${\cal T}_j$ and
the quasiuniformity of ${\cal T}_0$, such that
\begin{equation} \label{ineq:localRef1}
\frac{ \size(\tau) } { \size(\tau^\prime) } \leq c,~~~
\forall \tau, \tau^\prime \in {\cal T}_j,~~~
\tau \cap \tau^\prime \neq \emptyset.
\end{equation}
\end{lemma}
\begin{proof}
$\tau$ and $\tau^\prime$ are either face-adjacent ($d$ vertices),
edge-adjacent ($d-1$ vertices), or vertex-adjacent, and are handled by
(\ref{ineq:face-adjacent}), (\ref{ineq:edge-adjacent}),
(\ref{ineq:vertex-adjacent}), respectively.
\begin{eqnarray*}
\frac{ \size(\tau) } { \size(\tau^\prime) }
& \leq & c~2^{ |L(\tau) - L(\tau^\prime)| }~
\frac{ \size(\bar{\tau}) } { \size(\bar{\tau^\prime}) }~~~
(\mbox{by (\ref{ineq:childFather1}))} \\
& \leq & c~2^{\max\{1,E,V\}}~\gamma^{(0)}~~~
(\mbox{by (\ref{ineq:face-adjacent}), (\ref{ineq:edge-adjacent}),
(\ref{ineq:vertex-adjacent}) and quasiuniformity of ${\cal T}_0$)}
\end{eqnarray*}
\end{proof}

\subsection{$L_2$-stable Riesz basis} \label{sec:stableRieszBasis}

Since patchwise quasiuniformity is established by (\ref{ineq:localRef1}),
we can now take the first step in establishing the norm equivalence in 
section~\ref{sec:optBPX}. In other words, our motivation is to form a stable 
basis in the following sense~\cite{Os94book}.
\begin{equation} \label{id:stableBasisOswald}
\| \sum_{x_i \in {\cal N}_j} u_i \phi_i^{(j)} \|_{L_2(\Omega)} \eqsim
\|\{ {\rm volume}^{1/2}
(\supp~\phi_i^{(j)})~u_i\}_{x_i \in {\cal N}_j} \|_{l_2}.
\end{equation}
The basis stability (\ref{id:stableBasisOswald}) will then guarantee
that the Bernstein inequality (\ref{ineq:Bernstein}) holds. For a
stable basis, functions with small supports have to be augmented by an
appropriate scaling so that $\|\phi_i^{(j)}\|_{L_2(\Omega)}$ remains
roughly the same for all basis functions. This is reflected in ${\rm
volume}(\supp~\phi_i^{(j)})$ by defining:
\begin{equation} \label{defn:Ljv}
L_{j,i} = \min \{ L(\tau): \tau \in {\cal T}_j,~x_i \in \tau \}.
\end{equation}
Then
$$
{\rm volume} (\supp~\phi_i^{(j)}) \eqsim 2^{-d L_{j,i}}.
$$
We prefer to use an equivalent notion of basis stability;
a basis is called {\em $L_2$-stable Riesz basis} if:
\begin{equation} \label{id:sBasisOswaldEquiv}
\| \sum_{x_i \in {\cal N}_j} \hat{u}_i 
\hat{\phi}_i^{(j)}\|_{L_2(\Omega)} \eqsim
\|\{ \hat{u}_i\}_{x_i \in {\cal N}_j}\|_{l_2},
\end{equation}
where $\hat{\phi}_i^{(j)}$ denotes the scaled basis, and the
relationship between (\ref{id:stableBasisOswald}) and
(\ref{id:sBasisOswaldEquiv}) is given as follows:
\begin{equation} \label{scaledBasis}
\hat{\phi}_i^{(j)} = 2^{d/2 L_{j,i}}~\phi_i^{(j)},~~~
\hat{u}_i = 2^{-d/2 L_{j,i}}~u_i,~~~x_i \in {\cal N}_j.
\end{equation}
Then (\ref{id:sBasisOswaldEquiv}) forms the sufficient condition to establish
the  Bernstein inequality (\ref{ineq:Bernstein}).
This crucial property helps us to prove Theorem~\ref{thm:basisStab}.

\begin{remark}
The analysis is done purely with basis functions, completely
independent of the underlying mesh geometry. Furthermore,
our construction works for any $d$-dimensional setting with the
scaling (\ref{scaledBasis}). However, it is not clear how to define
face-adjacency relations for $d>3$.  If such relations can be defined
through some topological or geometrical abstraction, then our
framework naturally extends to $d$-dimensional local refinement
strategies, and hence the optimality of the BPX and WHB preconditioners
can be guaranteed in $\Re^d,~d \geq 1$. One such generalization was given by
Brandts-Korotov-Krizek in~\cite{BrKoKr04} and in the references therein.
\end{remark}

\section{Local smoothing computational complexity} \label{sec:localComplexity}

In~\cite{BoErKo}, the smoother is chosen to act on the local space
$$\tilde{{\cal S}}_j = {\rm span} \left[
\bigcup \{ \phi_i^{(j)} \}_{i=N_{j-1}+1}^{N_j}
\bigcup \{ \phi_i^{(j)} \neq \phi_i^{(j-1)} \}_{i=1}^{N_{j-1}} 
\right].
$$
Other choices for $\tilde{{\cal N}}_j$  are also possible; e.g., DOF 
which intersect the refinement region $\Omega_j$~\cite{AkBoHo03,BrPa93}.
The only restriction is that $\tilde{{\cal N}}_j \subset \Omega_j$.
For this particular choice,
$
\tilde{{\cal N}}_j = \{i=N_{j-1}+1, \ldots, N_j \}
\bigcup \{ i: \phi_i^{(j)} \neq \phi_i^{(j-1)},~~i=1, \ldots, N_{j-1}\},
$
the following result from~\cite{BoErKo} establishes a bound for the number
of nodes used for smoothing (those created in $\Omega_j$ by the
BEK procedure) so that the BPX preconditioner has provably optimal (linear)
computational complexity per iteration.
\begin{lemma}
The total number of nodes used for smoothing satisfies the bound:
\begin{equation} \label{ineq:smoothBound}
\sum_{j=0}^J \tilde{N}_j \leq \frac{5}{3}N_J - \frac{2}{3}N_0.
\end{equation}
\end{lemma}
\begin{proof}
See~\cite[Lemma 1]{BoErKo}.
\end{proof}

A similar result for 2D red-green refinement was
given by Oswald~\cite[page 95]{Os94book}.
In the general case of local smoothing operators which involve
smoothing over newly created basis functions plus some additional set of
local neighboring basis functions, one can extend the arguments
from~\cite{BoErKo} and~\cite{Os94book} using shape regularity.

\section{Establishing optimality of the BPX norm equivalence} \label{sec:optBPX}
In this section, we extend the Dahmen-Kunoth framework to three spatial
dimensions; the extension closely follows the original work. However, the 
general case for $d \geq 1$ spatial dimensions is not in the 
literature, and therefore we present it below.

For linear $g$, the element mass matrix gives rise to the following useful 
formula.
\begin{equation} \label{form:massIntegral}
\|g\|_{L_2(\tau)}^2 = \frac{{\rm volume}(\tau)}{(d+1)(d+2)}
~( \sum_{i=1}^{d+1} g(x_i)^2+ [ \sum_{i=1}^{d+1} g(x_i) ]^2 ),
\end{equation}
where,~$i=1, \ldots, d+1$ and $x_i$ is a vertex of $\tau$, $d=2,3$.
In view of (\ref{form:massIntegral}), we have that
$$
\|\hat{\phi}_i^{(j)}\|_{L_2(\Omega)}^2 = 2^{d L_{j,i}}~
\frac{{\rm volume(\supp }~\hat{\phi}_i^{(j)})}{(d+1)(d+2)}.
$$

Since the $\min$ in (\ref{defn:Ljv}) is attained, there exists
at least one $\tau \in \supp~\hat{\phi}_i^{(j)}$ such that $L(\tau) = L_{j,i}$.
By (\ref{ineq:childFather1}) we have
\begin{equation} \label{id:1}
2^{L_{j,i}} \eqsim \frac{ \size(\tau) }{ \size(\bar{\tau})}.
\end{equation}
Also,
\begin{equation} \label{id:2}
{\rm volume(supp}~\hat{\phi}_i^{(j)}) \eqsim \sum_{i=1}^E \size^d(\tau_i),~~~
\tau_i \in \supp~\hat{\phi}_i^{(j)}.
\end{equation}
By (\ref{ineq:localRef1}), we have
\begin{equation} \label{id:3}
\size(\tau_i) \eqsim \size(\tau).
\end{equation}
Combining (\ref{id:2}) and (\ref{id:3}), we conclude
\begin{equation} \label{id:4}
{\rm volume(supp}~\hat{\phi}_i^{(j)}) \eqsim E~\size^d(\tau).
\end{equation}
Finally then, (\ref{id:1}) and (\ref{id:4}) yield
$$
2^{d L_{j,i}} {\rm volume(supp}~\hat{\phi}_i^{(j)}) \eqsim
E~ \frac{1}{\size^d(\bar{\tau})}.
$$
$E$ is a uniformly bounded constant by shape regularity. One can view the size 
of any tetrahedron in ${\cal T}_0$, in particular size of $\bar{\tau}$, as a 
constant. The reason is the following: A.\ref{assump:rg4} and A.\ref{assump:rg5} force
every tetrahedron at any ${\cal T}_j$ to be geometrically similar to some 
tetrahedron in ${\cal T}_0$ or to a tetrahedron arising from an irregular refinement
of some tetrahedron in ${\cal T}_0$, hence, to some tetrahedron of a fixed finite 
collection. Combining the two arguments above, we have established that
\begin{equation} \label{basisFunction1}
\|\hat{\phi}_i^{(j)}\|_{L_2(\Omega)} \eqsim 1,~~~x_i \in {\cal N}_j.
\end{equation}

Let $g=\sum_{x_i \in {\cal N}_j} \hat{u}_i \hat{\phi}_i^{(j)} \in {\cal S}_j$.
For any $\tau \in {\cal T}_j$ we have that
\begin{equation} \label{ineq:Dahmen5.8}
\|g\|_{L_2(\tau)}^2 \leq c~\sum_{x_i \in {\cal N}_{j,\tau}} |\hat{u}_i|^2
\|\hat{\phi}_i^{(j)}\|_{L_2(\Omega)}^2,
\end{equation}
where  ${\cal N}_{j,\tau} = \{ x_i \in {\cal N}_j: x_i \in \tau \}$, which is
uniformly bounded in $\tau \in {\cal T}_j$ and $j \in \Na_0$. By the
scaling (\ref{scaledBasis}), we get equality in the estimate below.
The inequality is a standard inverse inequality where one bounds
$g(x_i)$ using formula (\ref{form:massIntegral}) and by handling the volume in
the formula by  (\ref{ineq:childFather1}):
\begin{equation} \label{ineq:Dahmen5.9}
|\hat{u}_i|^2 = 2^{-d L_{j,i}} |g(x_i)|^2
\leq c~ 2^{-d L_{j,i}} 2^{d L_{j,i}} \|g\|_{L_2(\tau)}^2.
\end{equation}
Now, we are ready to establish that our basis is an $L_2$-stable Riesz basis
as in (\ref{id:sBasisOswaldEquiv}). This is achieved by simply summing up over 
$\tau \in {\cal T}_j$ in (\ref{ineq:Dahmen5.8}) and (\ref{ineq:Dahmen5.9})
and using (\ref{basisFunction1}).
$L_2$ stability in the Riesz sense allows us to establish the Bernstein 
inequality (\ref{ineq:Bernstein}).
\begin{lemma} \label{lemma:BernsteinHolds}
For the scaled basis (\ref{scaledBasis}), the Bernstein inequality (\ref{ineq:Bernstein})
holds for $\beta = 3/2$
\end{lemma}
\begin{proof}
(\ref{basisFunction1}) with (\ref{ineq:Dahmen5.8}) and (\ref{ineq:Dahmen5.9}) assert
that the scaled basis (\ref{scaledBasis}) is stable in the sense of
(\ref{id:sBasisOswaldEquiv}). 
Hence, (\ref{ineq:Bernstein}) holds by~\cite[Theorem 4]{Os94book}.
Note that the proof actually works independently of the spatial dimension.
\end{proof}

\section{Lower bound in the norm equivalence} \label{sec:reason}

The Jackson inequality for Besov spaces is defined as follows:
\begin{equation} \label{Jackson2}
\inf_{g \in {\cal S}_J} \| f - g \|_{L_p}
 \leq c ~\omega_\alpha (f, 2^{-J})_p,~~~f \in L_p(\Omega),
\end{equation}
where $c$ is a constant independent of $f$ and $J$, and $\alpha$ is an
integer.  In the uniform refinement setting, (\ref{Jackson2}) is used
to obtain the lower bound in~(\ref{mainNormEquiv}).  However, in the
local refinement setting, (\ref{Jackson2}) holds only for functions
whose singularities are somehow well-captured by the mesh
geometry. For instance, if a mesh is designed to pick up the
singularity at $x=0$ of $y = 1/x$, then on the same mesh we will not
be able to recover a singularity at $x=1$ of $y= 1/(x-1)$. Hence the
Jackson inequality (\ref{Jackson2}) cannot hold in a general setting,
i.e. for $f \in W_p^k$. In order to get the lower bound
in~(\ref{mainNormEquiv}), we focus on estimating $\theta_J$ directly,
as in~\cite{DaKu92} for the 2D setting.

To begin we borrow the quasi-interpolant construction
from~\cite{DaKu92}, extending it to the three-dimensional setting.
Let $\tau \in {\cal T}_j$ be a tetrahedron with vertices
$x_1,x_2,x_3,x_4$.  Clearly the restrictions of $\hat{\phi}_i^{(j)}$
to $\tau$ are linearly independent over $\tau$ where $x_i \in
\{x_1,x_2,x_3,x_4\}$. Then, there exists a unique set of linear
polynomials $\psi_1^\tau, \psi_2^\tau, \psi_3^\tau, \psi_4^\tau$ such
that
\begin{equation} \label{id:biorthoPoly}
\int_\tau \hat{\phi}_k^{(j)} (x,y,z) \psi_l^\tau (x,y,z)
dx dy dz = \delta_{kl},~~~x_k, x_l \in \{x_1,x_2,x_3,x_4\}.
\end{equation}
For $x_i \in {\cal N}_j$ and $\tau \in {\cal T}_j$, define a function
for $x_i \in \tau$

\begin{equation} \label{defn:xi}
M_i^{(j)} (x,y,z) =
\left\{ \begin{array}{ll}
\frac{1}{E_i} \psi_i^\tau (x,y,z), & (x,y,z) \in \tau \\
0,    & (x,y,z) \not\in \supp~\hat{\phi}_i^{(j)}
\end{array} \right.
,\end{equation}
where $E_i$ is the number of tetrahedra in ${\cal T}_j$ in
$\supp~\hat{\phi}_i^{(j)}$. By (\ref{id:biorthoPoly}) and
(\ref{defn:xi}), we obtain
\begin{equation} \label{id:biorthoBasis}
(M_k^{(j)}, \hat{\phi}_l^{(j)}) = \int_{\Omega} M_k^{(j)}(x,y,z) 
\hat{\phi}_l(x,y,z)~dxdydz = \delta_{kl},~~~x_k, x_l \in {\cal N}_j.
\end{equation}

We can now define a quasi-interpolant, in fact a {\em projection} onto
${\cal S}_j$, such that
\begin{equation} \label{defn:quasi-interpol}
(\tilde{Q}_j f)(x,y,z) = \sum_{x_i \in {\cal N}_j}
(f, M_i^{(j)}) \hat{\phi}_i^{(j)} (x,y,z).
\end{equation}
As remarked earlier, due to~(\ref{defn:xi}) the slice operator term
$\tilde{Q}_j - \tilde{Q}_{j-1}$ will vanish outside the refined set
$\Omega_j$ defined in~(\ref{eqn:omega}).
One can easily observe by (\ref{basisFunction1}) and (\ref{id:biorthoBasis})
that
\begin{equation} \label{id:biortho1}
\|M_i^{(j)}\|_{L_2(\Omega)} \eqsim 1,~~~x_i \in {\cal N}_j, ~j \in \Na_0.
\end{equation}

Letting  $\Omega_{j,\tau} = \bigcup \{ \tau^\prime \in {\cal T}_j :~
\tau \cap \tau^\prime \neq \emptyset \}$, we can conclude from
(\ref{basisFunction1}) and (\ref{id:biortho1}) that
\begin{equation} \label{ineq:quasiInterpol}
\|\tilde{Q}_j f\|_{L_2(\tau)} =
\| \sum_{x_k \in {\cal N}_{j,\tau}} (f,M_{l}^{(j)})\hat{\phi}_k^{(j)} \|_{L_2(\tau)}
\leq c \|f\|_{L_2(\Omega_{j,\tau})}.
\end{equation}

We define now a subset of the triangulation where the refinement activity
stops, meaning that all tetrahedra in ${\cal T}_j^\ast,~j \leq m$ also belong to
${\cal T}_m$:
\begin{equation} \label{defn:dinosaurSet}
{\cal T}_j^\ast = \{ \tau \in {\cal T}_j:~L(\tau)<j,
~\Omega_{j,\tau} \cap \tau^\prime = \emptyset,
~\forall \tau^\prime \in {\cal T}_j~{\rm with}~L(\tau^\prime) = j \}.
\end{equation}
Due to the local support of the dual basis functions $M_i^{(j)}$ and 
the fact that $\tilde{Q}_j$ is a projection, one gets for $g \in {\cal S}_J$:
\begin{equation} \label{id:q-interFixes2}
\| g- \tilde{Q}_j g  \|_{L_2(\tau)} = 0,~~~\tau \in {\cal T}_j^\ast.
\end{equation}

Since $\tilde{Q}_j$ is a projection onto linear finite element space,
it fixes polynomials of degree at most $1$ (i.e. $\Pi_1(\Re^3))$. Using
this fact  and (\ref{ineq:quasiInterpol}), we arrive:
\begin{eqnarray}
\| g- \tilde{Q}_j g  \|_{L_2(\tau)} & \leq & \| g - P \|_{L_2(\tau)} +
\| \tilde{Q}_j (P-g) \|_{L_2(\tau)} \nonumber \\
& \leq & c ~ \| g - P \|_{L_2(\Omega_{j,\tau})},\quad
\tau \in {\cal T}_j \setminus {\cal T}_j^\ast.  \label{ineq:q-interNeedsWhitney}
\end{eqnarray}

We would like to bound the right hand side of (\ref{ineq:q-interNeedsWhitney})
in terms of a modulus of smoothness in order to reach a Jackson-type inequality.
Following~\cite{DaKu92}, we utilize a modified modulus of smoothness
for $f \in L_p(\Omega)$
$$
\tilde{\omega}_k (f,t,\Omega)_p^p = t^{-s} \int_{[-t,t]^s}
\|\Delta_h^k f \|_{L_p(\Omega_{k,h})}^p~dh.
$$
They can be shown to be equivalent: 
$$
\tilde{\omega}_{k+1} (f,t,\Omega)_p \eqsim \omega_{k+1} (f,t,\Omega)_p.
$$
The equivalence in the one-dimensional setting can be found 
in~\cite[Lemma 5.1]{DeLo93}.

For $\tau$ a simplex in $\Re^d$ and $t = \size(\tau)$,
a Whitney estimate shows that~\cite{DePo88,Os90,StOs78}
\begin{equation} \label{ineq:Whitney}
\inf_{P \in \Pi_k (\Re^d)} \|f-P\|_{L_p(\tau)}
\leq c \tilde{\omega}_{k+1}(f,t,\tau)_p,
\end{equation}
where $c$ depends only on the smallest angle of $\tau$ but not on $f$
and $t$. The reason why $\tilde{Q}_j$ works well for tetrahedralization
in 3D is the fact that the Whitney estimate (\ref{ineq:Whitney}) remains valid
for any spatial dimension. ${\cal T}_j \setminus {\cal T}_j^\ast$ is the
part of the tetrahedralization ${\cal T}_j$ where refinement is active at every
level. Then, in view of (\ref{ineq:localRef1})
$$
\size(\Omega_{j,\tau}) \eqsim 2^{-j},
~~~\tau \in {\cal T}_j \setminus {\cal T}_j^\ast .
$$
Taking the $\inf$ over $P \in \Pi_{1}(\Re^3)$ in
(\ref{ineq:q-interNeedsWhitney}) and using the Whitney estimate
(\ref{ineq:Whitney}) we conclude
$$
\| g- \tilde{Q}_j g  \|_{L_2(\tau)} \leq c \tilde{\omega}_{2} (g,2^{-j},
\Omega_{j,\tau})_2.
$$
Recalling (\ref{id:q-interFixes2}) and summing over
$\tau \in {\cal T}_j \setminus {\cal T}_j^\ast$ gives rise to
$$
\| g- \tilde{Q}_j g \|_{L_2(\Omega)}
   \leq c \tilde{\omega}_2 (g,2^{-j},\Omega)_2
   \leq \tilde{c}~\omega_2 (g,2^{-j},\Omega)_2,
$$
where we have switched from the modified modulus of smoothness
to the standard one.  Since $Q_j$ is an orthogonal projection, we
have the following:
$$ \| g - Q_j g \| \leq \| g - \tilde{Q}_j g \|.$$
Using the above inequality with (\ref{def:lowerBound}) one then has
\begin{equation} \label{mainResult}
v_J = O(1),~~~J \rightarrow \infty.
\end{equation}

\section{The WHB preconditioner} \label{sec:WHBprecond}

In local refinement, HB methods enjoy an optimal complexity of 
$O(N_j- N_{j-1})$ per iteration per level (resulting in $O(N_J)$
overall complexity per iteration) by only using degrees of freedom
(DOF) corresponding to ${\cal S}_j^f$.  However, HB methods suffer
from suboptimal iteration counts or equivalently suboptimal 
condition number. The BPX decomposition 
${\cal S}_j = {\cal S}_{j-1} \oplus (Q_{j} -Q_{j-1}) {\cal S}_j$
gives rise to basis functions which are not locally supported, but
they decay rapidly outside a local support region. This allows for
locally supported approximations, and in addition the WHB
methods~\cite{VaWa2,VaWa1,VaWa3} can be viewed as an approximation of
the wavelet basis stemming from the BPX decomposition~\cite{Ja92}.  A
similar wavelet-like multilevel decomposition approach was taken
in~\cite{St97}, where the orthogonal decomposition is formed by a
discrete $L_2$-equivalent inner product. This approach utilizes the
same BPX two-level decomposition~\cite{St95,St97}. The WHB 
preconditioner is defined as follows:
\begin{equation} \label{id:genericHB}
H u := \sum_{j=0}^J
2^{j(d-2)} \sum_{i \in {\cal N}_j^f} (u,\psi_i^{(j)}) \psi_i^{(j)},
\end{equation}
where $\psi_i^{(j)} = (\tilde{Q}_j - \tilde{Q}_{j-1}) \phi_i^{(j)}$.
The WHB preconditioner uses the modified basis (where as the BPX preconditioner
uses the standard nodal basis) where the projection operator used is defined 
as in (\ref{Wk}).  In the WHB setting, these operators are chosen to satisfy
the following three properties~\cite{AkHo05-techReportII}:
\begin{eqnarray}
\tilde{Q}_j~|_{{\cal S}_j} & = & I, \label{id:prop1} \\
\tilde{Q}_j \tilde{Q}_k & = & \tilde{Q}_{\min\{j,k\}}, \label{id:prop2} \\
\|(\tilde{Q}_j - \tilde{Q}_{j-1}) u^{(j)} \|_{L_2} & \eqsim & \|u^{(j)} \|_{L_2},~~~u^{(j)} \in
(I_j - I_{j-1}) {\cal S}_j. \label{id:prop3}
\end{eqnarray}

As indicated in (\ref{subset:localSpace}), the WHB smoother acts on
only the fine DOF, i.e. ${\cal N}_j^f$, and hence is an approximation
to fine-fine discretization operator; 
$A_{ff}^{(j)}: \calg{S}_j^f \rightarrow \calg{S}_j^f$, 
where $\calg{S}_j^f := (\tilde{Q}_j- \tilde{Q}_{j-1}) {\cal S}_j$
and $f$ stands for fine. On the other hand, the BPX smoother acts on 
a slightly bigger set than fine DOF,
${\cal N}_j^f \subseteq \tilde{\calg{N}}_j$ typically, union of fine
DOF and their corresponding coarse fathers.

The WHB preconditioner introduced in~\cite{VaWa2,VaWa1} is, in some
sense, the best of both worlds.  While the condition number of the HB
preconditioner is stabilized by inserting $Q_j$ in the definition of
$\tilde{Q}_j$, somehow employing the operators $I_j - I_{j-1}$ at the same
time guarantees optimal computational and storage cost per iteration.
The operators which will be seen to meet both goals at the same time
are:
\begin{equation} \label{Wk}
\tilde{Q}_k = \prod_{j=k}^{J-1} I_j + Q_j^a (I_{j+1} - I_j),
\end{equation}
with $\tilde{Q}_J = I$. The exact $L_2$-projection $Q_j$ is replaced by a computationally
feasible approximation $Q_j^a: L_2 \rightarrow {\cal S}_j$. To control the approximation
quality of $Q_j^a$, a small fixed tolerance $\gamma$ is introduced:
\begin{equation} \label{ineq:gamma}
\| (Q_j^a-Q_j)u \|_{L_2} \leq \gamma \|Q_ju \|_{L_2}, ~~~ \forall u \in
L_2(\Omega).
\end{equation}
In the limiting case $\gamma=0$, $\tilde{Q}_k$ reduces to
the exact $L_2$-projection on ${\cal S}_J$ by (\ref{id:prop1}):
$$
\tilde{Q}_k = Q_k~~ I_{k+1} Q_{k+1} \ldots I_{J-1} Q_{J-1}~~ I_J
= Q_k Q_{k+1}  \ldots Q_{J-1} = Q_k.
$$
Following~\cite{VaWa2,VaWa1},
the properties (\ref{id:prop1}), (\ref{id:prop2}), and (\ref{id:prop3})
can be verified for $\tilde{Q}_k$ as follows:

$\bullet$ Property (\ref{id:prop1}): Let $u^{(k)} \in {\cal S}_k$. Since
$(I_{j+1} - I_j)u^{(k)}=0$ and $I_j u^{(k)} = u^{(k)}$ for $k \leq j$, then
$[I_j + Q_j^a (I_{j+1} - I_j)] (u^{(k)}) = u^{(k)}$,  verifying (\ref{id:prop1})
for $\tilde{Q}_k$. It also implies
\begin{equation} \label{Pik:projection}
\tilde{Q}_k^2 = \tilde{Q}_k.
\end{equation}

$\bullet$ Property (\ref{id:prop2}): Let $k \leq l$, then by (\ref{Pik:projection})
\begin{equation} \label{verify1:id:prop2}
\tilde{Q}_k \tilde{Q}_l = [(I_k + Q_k^a (I_{k+1} - I_k)) \ldots
(I_{l-1} + Q_{l-1}^a (I_l - I_{l-1}))~\tilde{Q}_l] \tilde{Q}_l = \tilde{Q}_k.
\end{equation}
Since
$\tilde{Q}_k u \in {\cal S}_k$ and ${\cal S}_k \subset {\cal S}_l$, then by (\ref{id:prop1}) we have
\begin{equation} \label{verify2:id:prop2}
\tilde{Q}_l (\tilde{Q}_k u) = \tilde{Q}_k u.
\end{equation}
Finally, (\ref{id:prop2}) then follows from (\ref{verify1:id:prop2}) and (\ref{verify2:id:prop2}).

$\bullet$ Property (\ref{id:prop3}): This is an implication of Lemma
\ref{lemma:prop3Established}.

For an overview, we list the corresponding slice spaces for the
preconditioners of interest:
$$
\begin{array}{llll}
\mbox{HB:}  & {\cal S}_j^f & = & (I_j - I_{j-1}){\cal S}_j, \\
\mbox{BPX:} & {\cal S}_j^f & = & (Q_j - Q_{j-1}){\cal S}_j,\\
\mbox{WHB:} & {\cal S}_j^f & = & (\tilde{Q}_{j} -\tilde{Q}_{j-1}){\cal S}_j =
(I - Q_{j-1}^a)(I_j-I_{j-1}){\cal S}_j,~
\tilde{Q}_j~\mbox{as in (\ref{Wk}).}
\end{array}
$$

The WHB smoother only acts on the fine DOF. Then, in the generic multilevel 
preconditioner notation, the WHB preconditioner can be written in the
following form:
\begin{equation} \label{def:precondW}
B u := \sum_{j=0}^J B_{ff}^{(j)^{-1}} (\tilde{Q}_j - \tilde{Q}_{j-1}) u.
\end{equation}
$B_{ff}$ is chosen to be a spectrally equivalent operator to fine-fine
discretization operator $A_{ff}^{(j)}$.  Since the smoother and
property (\ref{id:prop3}) both rely on a well-conditioned
$A_{ff}^{(j)}$, we discuss this next.

\subsection{Well-conditioned $A_{ff}^{(j)}$} \label{subsec:wellCond}
The lemma below is essential to extend the existing results for quasiuniform
meshes~\cite[Lemma 6.1]{VaWa2} or~\cite[Lemma 2]{VaWa1} to the 
locally refined ones. ${\cal S}_j^{(f)} = (I_j - I_{j-1}) {\cal S}_j$ denotes
the HB slice space.

\begin{lemma} \label{lemma:prop3Established}
Let ${\cal T}_j$ be constructed by the local refinements under consideration.
Let ${\cal S}_j^f = (I- \tilde{Q}_{j-1}){\cal S}_j^{(f)}$ be the modified
hierarchical subspace where $\tilde{Q}_{j-1}$ is any $L_2$-bounded operator. Then,
there are constants $c_1$ and $c_2$ independent of $j$ such that
\begin{equation} \label{ineq:mHBequi}
c_1 \|\phi^f\|_X^2 \leq \|\psi^f\|_X^2 \leq c_2 \|\phi^f\|_X^2,~~X=H^1,L_2,
\end{equation}
holds for any $\psi^f = (I- \tilde{Q}_{j-1}) \phi^f \in {\cal S}_j^f$ with
$\phi^f \in {\cal S}_j^{(f)}$.
\end{lemma}
\begin{proof}
The Cauchy-Schwarz like inequality~\cite{BaDu80} is central to the proof: There
exists $\delta \in (0,1)$ independent of the mesh size or level $j$ such that
\begin{eqnarray} \label{equivC-S}
\hspace{-1cm}
(1- \delta^2) (\nabla \phi^f, \nabla \phi^f) & \leq &
(\nabla ( \phi^c + \phi^f),\nabla(\phi^c + \phi^f)),~~~ \forall
\phi^c \in {\cal S}_{j-1}, \phi^f \in {\cal S}_j^{(f)}. \\
\label{ineq:lemmaHelps}
\hspace{-1cm}
(1- \delta^2) \| \phi^f \|_{L_2}^2 & \leq & c | \phi^c + \phi^f |_{H^1}^2~~
(\mbox{by Poincare inequality and (\ref{equivC-S})}). 
\end{eqnarray}
Combining (\ref{equivC-S}) and (\ref{ineq:lemmaHelps}):
$
(1- \delta^2) \|\phi^f \| _{H^1} ^2 \leq
\|\phi^c + \phi^f \|_{H^1} ^2.
$
Choosing $ \phi^c = - \tilde{Q}_{j-1}  \phi^f$, we get the lower bound:
$
(1- \delta^2) \|\phi^f \| _{H^1} ^2 \leq \| \psi ^f \|_{H^1}^2.
$

Let $\Omega_j^f$  denote the support of basis functions corresponding to 
${\cal N}_j^f$. Due to nested refinement, triangulation on $\Omega_j^f$ is
quasiuniform. One can analogously introduce a triangulation hierarchy
where all the simplices are exposed to uniform refinement: 
$
{\cal T}_j^f := \{\tau \in {\cal T}_j: L(\tau) = j \} = {\cal T}_j|_{\Omega_j^f}.
$
Hence, ${\cal T}_j^f$ becomes a quasiuniform tetrahedralization and the 
inverse inequality holds for ${\cal S}_j^f$. To derive the upper bound:
The right scaling is obtained by father-son size relation, and by the inverse 
inequalities and $L_2$-boundedness of $\tilde{Q}_{j-1}$, one gets
$$
\|\psi ^f\|_{H^1} ^2 \leq c_0 2 ^{2j} \|\psi ^f\|_{L_2} ^2 
\leq c_0 2 ^{2j} \left( 1+ \| \tilde{Q}_{j-1} \|_{L_2}\right)^2 \|\phi ^f\|_{L_2} ^2
\leq c 2^{2j} \|\phi ^f\|_{L_2} ^2.
$$
The slice space ${\cal S}_j^{(f)}$ is oscillatory. Then
there exists $c$ such that
$ \| \phi^f \|_{L_2}^2 \leq c 2^{-2j} \| \phi^f \|_{H^1}^2.$
Hence, $ \| \psi ^f \|_{H^1}^2 \leq c \| \phi^f\|_{H^1}^2$.
The case for $X=L_2$ can be established similarly.
\end{proof}

Using the above tools, one can establish that $A_{ff}^{(j)}$ is 
well-conditioned. Namely, 
\begin{equation} \label{ineq:WcondA_22}
c_1 2^{2j} \leq \lambda_{j,\min}^f \leq \lambda_{j,\max}^f \leq c_2 2^{2j},
\end{equation}
where $\lambda_{j,\min}^f$ and $\lambda_{j,\max}^f$ are the smallest and 
largest eigenvalues of $A_{ff}^{(j)}$, and $c_1$ are and $c_2$ both 
independent of $j$. For details see~\cite[Lemma 4.3]{VaWa2} 
or~\cite[Lemma 3]{VaWa1}.

\section{The fundamental assumption and WHB optimality}
\label{sec:assump-stabp}
As in the BPX splitting, the main ingredient in the WHB splitting is the
$L_2$-projection. Hence, the stability of the BPX splitting is still
important in the WHB splitting. The lower bound in the BPX norm equivalence
is the {\em fundamental assumption} for the WHB preconditioner. 
Utilizing a local projection $\tilde{Q}_j$, BPX lower bound was
verified earlier for 3D local red-green (BEK) refinement procedure.
The same result easily holds for the projection $Q_j$.  Dahmen and
Kunoth~\cite{DaKu92} verified BPX lower bound for the 2D red-green
refinement procedures.

Before getting to the stability result we remark that the existing perturbation
analysis of WHB is one of the primary insights in~\cite{VaWa2,VaWa1}.
Although not observed in~\cite{VaWa2,VaWa1}, the result
does not require substantial modification for locally refined meshes.
Let
$e_j := (\tilde{Q}_j  - Q_j) u$ be the error, then the following holds.
\begin{lemma}
Let $\gamma$ be as in (\ref{ineq:gamma}). There exists an absolute $c$ 
satisfying:
\begin{equation} \label{ineq:crucial}
\sum_{j=0}^J 2^{2j} \|e_j\|_{L_2}^2 \leq
c \gamma^2 \sum_{j=0}^J 2^{2j} \|(Q_j - Q_{j-1})u\|_{L_2}^2,~~~\forall u \in {\cal S}_J.
\end{equation}
\end{lemma}

\begin{proof}
\cite[Lemma 5.1]{VaWa2} or~\cite[Lemma 1]{VaWa1}.
\end{proof}

We arrive now at the primary result, which indicates that the
WHB slice norm is optimal on the class of locally refined meshes
under consideration.
\begin{theorem} \label{thm:basisStab}
If there exists sufficiently small $\gamma_0$ such that
(\ref{ineq:gamma}) is satisfied for $\gamma \in [0,\gamma_0)$, then
\begin{equation} \label{basisStab}
\|u\|_{\mbox{{\tiny {\rm WHB}}}}^2 = \sum_{j=0}^J 2^{2j} \|(\tilde{Q}_j - \tilde{Q}_{j-1})u\|_{L_2}^2
\eqsim \|u\|_{H^1}^2,~~~ u \in {\cal S}_J.
\end{equation}
\end{theorem}

\begin{proof}
Observe that
\begin{eqnarray} \label{id:DecompError}
(\tilde{Q}_j - \tilde{Q}_{j-1})u & = & (\tilde{Q}_j - Q_j)u - (\tilde{Q}_{j-1} - Q_{j-1})u + (Q_j - Q_{j-1})u\\
\nonumber        & = & e_j-e_{j-1}+(Q_j - Q_{j-1})u.
\end{eqnarray}
This gives
\begin{eqnarray*}
\sum_{j=0}^J 2^{2j} \|(\tilde{Q}_j - \tilde{Q}_{j-1})u \|_{L_2}^2 & \leq &
c \sum_{j=0}^J 2^{2j} \|(Q_j - Q_{j-1})u\|_{L_2}^2 +
c \sum_{j=0}^J 2^{2j} \| e_j \|_{L_2}^2 \\
& \leq & c(1+\gamma^2) \sum_{j=0}^J 2^{2j} \|(Q_j - Q_{j-1})u\|_{L_2}^2
~~\mbox{(using (\ref{ineq:crucial}))} \\
& \leq & c \|u\|_{H^1}^2.
\end{eqnarray*}
Let us now proceed with the upper bound.  The Bernstein inequality
(\ref{ineq:Bernstein}) holds for ${\cal S}_j$~\cite{Ak01phd,DaKu92}
for the local refinement procedures. Hence we are going to utilize an
inequality involving the Besov norm $\|\cdot\|_{B_{2,2}^1}$ which
naturally fits our framework when the moduli of smoothness is
considered in (\ref{ineq:Bernstein}). The following important
inequality holds, provided that (\ref{ineq:Bernstein})
holds~\cite[page 39]{Os94book}:
\begin{equation} \label{OswaldHelps}
\|u\|_{B_{2,2}^1}^2 \leq c \sum_{j=0}^J 2^{2j} \| u^{(j)} \|_{L_2}^2,
\end{equation}
for any decomposition such that $u=\sum_{j=0}^J u^{(j)},~~u^{(j)} \in {\cal S}_j$,
in particular for $u^{(j)} = (\tilde{Q}_j - \tilde{Q}_{j-1})u$. Then the upper bound holds due to $H^1(\Omega) \cong B_{2,2}^1(\Omega).$
\end{proof}

\begin{remark}
The following equivalence is used for the upper bound in the proof of Theorem
\ref{thm:basisStab} on uniformly refined meshes~\cite[Lemma 4]{VaWa1}.
$$
c_1 \|u\|_{H^1}^2 \leq \inf_{ u=\sum_{j=0}^J u^{(j)},~u^{(j)} \in {\cal S}_j }
\sum_{j=0}^J 2^{2j} \|u^{(j)}\|_{L_2}^2 \leq c_2 \|u\|_{H^1}^2.
$$ Let us emphasize that the left hand side holds in the presence of
the Bernstein inequality (\ref{ineq:Bernstein}), and the right hand side
holds in the simultaneous presence of Bernstein and Jackson
inequalities. However, the Jackson inequality cannot hold under local
refinement procedures (cf. counter example in section~\ref{sec:reason}). 
That is why we can utilize only the left hand side of the above
equivalence as in (\ref{OswaldHelps}).
\end{remark}

Now, we have all the required estimates at our disposal to
establish the optimality of WHB preconditioner for 2D/3D red-green 
refinement procedures for $p \in L_\infty(\Omega)$. We would like to
emphasize that our framework supports any spatial dimension $d \geq 1$,
provided that the necessary geometrical abstractions are in place.
\begin{theorem}
If BPX lower bound holds and if there exists sufficiently small $\gamma_0$
such that (\ref{ineq:gamma}) is satisfied for
$\gamma \in (0,\gamma_0)$, then for $B$ in (\ref{def:precondW}):
$$
(Bu, u) \eqsim \|u\|_{H^1}^2.
$$
\end{theorem}
\begin{proof}
$B_{ff}^{(j)}$ is spectrally equivalent to $A_{ff}^{(j)}$. Since
$A_{ff}^{(j)}$ is a well-conditioned matrix, using
(\ref{ineq:WcondA_22}) it is spectrally equivalent to $2^{2j} I$.  The
result follows from Theorem \ref{thm:basisStab}.
\end{proof}

An extension to multiplicative WHB preconditioner is also possible under 
additional assumptions. These results will not be reported here.

\section{$H^1$-stable $L_2$-projection} \label{sec:H1StableL2}

The involvement of $\tilde{Q}_j$ in the multilevel decomposition
makes it the most crucial element in the stabilization.
We then come to the central question: Which choice of $\tilde{Q}_j$
can provide an optimal preconditioner? 
The following theorem sets a guideline for picking $\tilde{Q}_j$.
It shows that $H^1$-stability of the $\tilde{Q}_j$ is actually
a {\em necessary condition} for obtaining an optimal preconditioner.

\begin{theorem} \label{thm:stabCriterion}
\cite{VaWa2,VaWa1}.
If $\tilde{Q}_j$ induces an optimal preconditioner, namely for 
$u \in {\cal S}_J$, 
$\sum_{j=0}^J 2^{2j} \|(\tilde{Q}_{j} - \tilde{Q}_{j-1})u\|_{L_2}^2 
\eqsim \|u\|_{H^1}^2$, 
then there exists an absolute constant $c$ such that
$$
\|\tilde{Q}_k u \|_{H^1} \leq c~\| u \|_{H^1},~~~\forall k \leq J.
$$
\end{theorem}
\begin{proof}
Using the multilevel decomposition and (\ref{id:prop2}), we get:\\
$
\tilde{Q}_k u = \sum_{j=0}^k (\tilde{Q}_j - \tilde{Q}_{j-1})u.
$
Since $\tilde{Q}_j$ induces an optimal preconditioner, there exist two
absolute constants $\sigma_1$ and $\sigma_2$:
\begin{equation} \label{defn:optPrecond}
\sigma_1 \|u\|_{H^1}^2 \leq 
\sum_{j=0}^J 2^{2j} \|(\tilde{Q}_{j} - \tilde{Q}_{j-1})u\|_{L_2}^2 \leq 
\sigma_2 \|u\|_{H^1}^2, ~~~\forall u \in \calg{S}_J.
\end{equation}
Using (\ref{defn:optPrecond}) for $\tilde{Q}_k u$:
$$
\|\tilde{Q}_k u \|_{H^1}^2 \leq \frac{1}{\sigma_1} \sum_{j=0}^k 2^{2j}
\|(\tilde{Q}_{j} - \tilde{Q}_{j-1})u\|_{L_2}^2
\leq \frac{1}{\sigma_1} \sum_{j=0}^J 2^{2j}
\|(\tilde{Q}_{j} - \tilde{Q}_{j-1})u\|_{L_2}^2
\leq  \frac{\sigma_2}{\sigma_1} \|u\|_{H^1}^2.
$$
\end{proof}

As a consequence of Theorem~\ref{thm:stabCriterion} we have
\begin{corollary}
$L_2$-projection restricted to ${\cal S}_j$, 
$Q_j |_{{\cal S}_j}: L_2 \rightarrow {\cal S}_j$, is $H^1$-stable on 2D and 3D
locally refined meshes by red-green refinement procedures.
\end{corollary}
\begin{proof}
Optimality of the BPX norm equivalence on the above locally refined meshes
was already established.
Application of Theorem \ref{thm:stabCriterion} with $Q_j$ proves the result.
Alternatively, the same result can be obtained through Theorem
\ref{thm:stabCriterion} applied to the WHB framework.
Theorem \ref{thm:basisStab} will  establish the optimality of the WHB 
preconditioner for the local refinement procedures. 
Hence, the operator $\tilde{Q}_j$ restricted to ${\cal S}_j$ is $H^1$-stable.
Since $\tilde{Q}_j$ is none other than $Q_j$ in the limiting case, we can also
conclude the $H^1$-stability of the $L_2$-projection.
\end{proof}

Our stability result appears to be the first {\em a priori} $H^1$-stability
for the $L_2$-projection on these classes of locally refined meshes.
$H^1$-stability of $L_2$-projection is guaranteed for the subset
${\cal S}_j$ of $L_2(\Omega)$, not for all of $L_2(\Omega)$.
This question is currently undergoing intensive study in the finite element 
and approximation theory community. The existing theoretical results, 
mainly in~\cite{BrPaSt00,Ca00}, involve {\em a posteriori} verification of 
somewhat complicated mesh conditions after refinement has taken place. If 
such mesh conditions are not satisfied, one has to redefine the mesh.
The mesh conditions mentioned require that the simplex
sizes do not change drastically between regions of refinement.
In this context, quasiuniformity in the support of a basis function becomes 
crucial. This type of local quasiuniformity  is usually called as 
{\em patchwise quasiuniformity}. Local quasiuniformity requires 
neighbor generation relations as in (\ref{ineq:face-adjacent}), 
neighbor size relations, and shape regularity
of the mesh. It was shown in~\cite{Ak01phd} that patchwise quasiuniformity 
holds also for 3D marked tetrahedron bisection~\cite{JoLi95} and for 2D newest
vertex bisection~\cite{Mi88,Se72}. These are then promising refinement 
procedures for which $H^1$-stability of the $L_2$-projection can be 
established.

\section{Conclusion}\label{sec:conclusion}

In this article, we examined the Bramble-Pasciak-Xu (BPX) norm
equivalence in the setting of local 3D mesh refinement.  In
particular, we extended the 2D optimality result for BPX due to Dahmen
and Kunoth to the local 3D red-green refinement procedure introduced
by Bornemann-Erdmann-Kornhuber (BEK).  The extension involved
establishing that the locally enriched finite element subspaces
produced by the BEK procedure allow for the construction of a scaled
basis which is formally Riesz stable.  This in turn rested entirely on
establishing a number of geometrical relationships between neighboring
simplices produced by the local refinement algorithms.  We remark
again that shape regularity of the elements produced by the refinement
procedure is insufficient to construct a stable Riesz basis for finite
element spaces on locally adapted meshes.  The $d$-vertex adjacency
generation bound for simplices in $\Re^d$ is the primary result
required to establish patchwise quasiuniformity for stable Riesz basis
construction, and this result depends delicately on the particular
details of the local refinement procedure rather than on shape
regularity of the elements.  We also noted in \S\ref{sec:3DRed-Green}
that these geometrical properties have been established
in~\cite{Ak01phd} for purely bisection-based refinement procedures
that have been shown to be asymptotically non-degenerate, and
therefore also allow for the construction of a stable Riesz basis.

We also examined the wavelet modified hierarchical basis (WHB) methods
of Vassilevski and Wang, and extended their original
quasiuniformity-based framework and results to local 2D and 3D
red-green refinement scenarios.  A critical step in the extension
involved establishing the optimality of the BPX norm equivalence for the
local refinement procedures under consideration, as established
in the first part of this article.  With the local refinement
extension of the WHB analysis framework presented here, we established
the optimality of the WHB preconditioner on locally refined meshes in
both 2D and 3D under the minimal regularity assumptions required for
well-posedness.  An interesting implication of the optimality of WHB
preconditioner was the {\em a priori} $H^1$-stability of the
$L_2$-projection.  Existing {\em a posteriori} approaches in the
literature dictate a reconstruction of the mesh if such conditions
cannot be satisfied.

The theoretical framework established here supports arbitrary spatial
dimension $d \geq 1$, and therefore allows extension of the optimality
results, the $H^1$-stability of $L_2$-projection results, and the
various supporting results to arbitrary $d \geq 1$.  We indicated
clearly which geometrical properties must be re-established to show
BPX optimality for spatial dimension $d \ge 4$.  All of the results
here require no smoothness assumptions on the PDE coefficients beyond
those required for well-posedness in $H^1$.

To address the practical computational complexity of implementable
versions of the BPX and WHB preconditioners, we indicated how the
number of degrees of freedom used for the smoothing step can be shown
to be bounded by a constant times the number of degrees of freedom
introduced at that level of refinement.  This indicates that practical
implementable versions of the BPX and WHB preconditioners for the
local 3D refinement setting considered here have provably optimal
(linear) computational complexity per iteration.  A detailed analysis
of both the storage and per-iteration computational complexity
questions arising with BPX and WHB implementations can be found in the
second article~\cite{AkBoHo03}.

%% file: ack.tex
\section*{Acknowledgments}
The authors thank R.~Bank, P.~Vassilevski, and J.~Xu for many
enlightening discussions.

%% file: m.bbl
\begin{thebibliography}{10}

\bibitem{Ak01phd}
{\sc B.~Aksoylu}, {\em Adaptive Multilevel Numerical Methods with Applications
  in Diffusive Biomolecular Reactions}, PhD thesis, Department of Mathematics,
  University of California, San Diego, La Jolla, CA, 2001.

\bibitem{AkBoHo03}
{\sc B.~Aksoylu, S.~Bond, and M.~Holst}, {\em
  \href{http://epubs.siam.org/sam-bin/getfile/SISC/articles/40767.pdf}{An
  odyssey into local refinement and multilevel preconditioning III:
  Implementation and numerical experiments}}, SIAM J. Sci. Comput., 25 (2003),
  pp.~478--498.

\bibitem{AkBoHo04-techReport}
\leavevmode\vrule height 2pt depth -1.6pt width 23pt, {\em
  \href{http://www.ices.utexas.edu/research/reports/2004/0450.pdf}
  {Implementation and theoretical aspects of the BPX preconditioner in the
  three dimensional local mesh refinement setting}}, tech. report, The
  University of Texas at Austin, Institute for Computational Engineering and
  Sciences, ICES Report 04-50, October, 2004.

\bibitem{AkHo05-techReportI}
{\sc B.~Aksoylu and M.~Holst}, {\em
  \href{http://www.ices.utexas.edu/research/reports/2005/0503.pdf} {An odyssey
  into local refinement and multilevel preconditioning I: Optimality of the BPX
  preconditioner}}, tech. report, The University of Texas at Austin, Institute
  for Computational Engineering and Sciences, ICES Report 05-03, January, 2005.

\bibitem{AkHo05-techReportII}
\leavevmode\vrule height 2pt depth -1.6pt width 23pt, {\em
  \href{http://www.ices.utexas.edu/research/reports/2005/0504.pdf} {An odyssey
  into local refinement and multilevel preconditioning II: Stabilizing
  hierarchical basis methods}}, tech. report, The University of Texas at
  Austin, Institute for Computational Engineering and Sciences, ICES Report
  05-04, January, 2005.

\bibitem{AkKhSc03}
{\sc B.~Aksoylu, A.~Khodakovsky, and P.~Schr{\"o}der}, {\em
  \href{http://epubs.siam.org/sam-bin/getfile/SISC/articles/43013
  .pdf}{Multilevel Solvers for Unstructured Surface Meshes}}, SIAM J. Sci.
  Comput., 26 (2005), pp.~1146--1165.

\bibitem{BaActa96}
{\sc R.~E. Bank}, {\em Hierarchical basis and the finite element method}, Acta
  Numerica,  (1996), pp.~1--43.

\bibitem{BaDu80}
{\sc R.~E. Bank and T.~Dupont}, {\em Analysis of a two-level scheme for solving
  finite element equations}, tech. report, Center for Numerical Analysis,
  University of Texas at Austin, 1980.
\newblock CNA--159.

\bibitem{BaDuYs88}
{\sc R.~E. Bank, T.~Dupont, and H.~Yserentant}, {\em The hierarchical basis
  multigrid method}, Numer. Math., 52 (1988), pp.~427--458.

\bibitem{Be95}
{\sc J.~Bey}, {\em Tetrahedral grid refinement}, Computing, 55 (1995),
  pp.~271--288.

\bibitem{BoErKo}
{\sc F.~Bornemann, B.~Erdmann, and R.~Kornhuber}, {\em Adaptive multilevel
  methods in three space dimensions}, Intl. J. for Numer. Meth. in Eng., 36
  (1993), pp.~3187--3203.

\bibitem{BoYs93}
{\sc F.~Bornemann and H.~Yserentant}, {\em A basic norm equivalence for the
  theory of multilevel methods}, Numer. Math., 64 (1993), pp.~455--476.

\bibitem{BrPa92}
{\sc J.~H. Bramble and J.~E. Pasciak}, {\em
  \href{http://www.jstor.org/cgi-bin/jstor/printpage/00255718/di981391/98p0046%
b/0.pdf?userID=80534414@utexas.edu/01cc99333c005012a2099&backcontext=results&c%
onfig=jstor&dowhat=Acrobat&0.pdf} {The analysis of smoothers for multigrid
  algorithms}}, Math. Comp., 58 (1992), pp.~467--488.

\bibitem{BrPa93}
\leavevmode\vrule height 2pt depth -1.6pt width 23pt, {\em
  \href{http://www.jstor.org/cgi-bin/jstor/printpage/00255718/di981395/98p0044%
7/0.pdf?userID=80534414@utexas.edu/01cc99333c005012a2099&backcontext=results&c%
onfig=jstor&dowhat=Acrobat&0.pdf} {New estimates for multilevel algorithms
  including the V-cycle}}, Math. Comp., 60 (1993), pp.~447--471.

\bibitem{BrPaSt00}
{\sc J.~H. Bramble, J.~E. Pasciak, and O.~Steinbach}, {\em
  \href{http://www.ams.org/mcom/2002-71-237/S0025-5718-01-01314-X/S0025-5718-0%
1-01314-X.pdf} {On the stability of the {$L^2$} projection in
  {$H^1(\Omega)$}}}, Math. Comp., 71 (2001), pp.~147--156.

\bibitem{BPX90}
{\sc J.~H. Bramble, J.~E. Pasciak, and J.~Xu}, {\em
  \href{http://www.jstor.org/cgi-bin/jstor/printpage/00255718/di970615/97p0002%
c/0.pdf?userID=80534414@utexas.edu/01cc99333c005011b5582&backcontext=results&c%
onfig=jstor&dowhat=Acrobat&0.pdf}{Parallel multilevel preconditioners}}, Math.
  Comp., 55 (1990), pp.~1--22.

\bibitem{BrKoKr04}
{\sc J.~Brandts, S.~Korotov, and M.~Krizek}, {\em {The strengthened
  Cauchy-Bunyakowski-Schwarz inequality for n-simplicial linear finite
  elements}}, SIAM J. Numer. Anal.,  (2004).
\newblock submitted.

\bibitem{Ca00}
{\sc C.~Carstensen}, {\em
  \href{http://www.ams.org/mcom/2002-71-237/S0025-5718-01-01316-3/S0025-5718-0%
1-01316-3.pdf}{Merging the Bramble-Pasciak-Steinbach and the Crouzeix-Thomée
  criterion for $H^1$-stability of the $L^2$-projection onto finite element
  spaces }}, Math. Comp., 71 (2001), pp.~157--163.

\bibitem{DaKu92}
{\sc W.~Dahmen and A.~Kunoth}, {\em Multilevel preconditioning}, Numer. Math.,
  63 (1992), pp.~315--344.

\bibitem{DeLo93}
{\sc R.~A. DeVore and G.~G. Lorentz}, {\em Constructive Approximation},
  Grundlehren der mathematischen Wissenschaften 303, Springer Verlag, Berlin
  Heidelberg, 1993.

\bibitem{DePo88}
{\sc R.~A. DeVore and V.~A. Popov}, {\em Interpolation of {B}esov spaces},
  Trans. Amer. Math. Soc., 305 (1988), pp.~397--414.

\bibitem{Hols2001a}
{\sc M.~Holst}, {\em Adaptive numerical treatment of elliptic systems on
  manifolds}, Advances in Computational Mathematics, 15 (2002), pp.~139--191.

\bibitem{Ja92}
{\sc S.~Jaffard}, {\em Wavelet methods for fast resolution of elliptic
  problems}, SIAM J. Numer. Anal., 29 (1992), pp.~965--986.

\bibitem{JoLi95}
{\sc B.~Joe and A.~Liu}, {\em Quality local refinement of tetrahedral meshes
  based on bisection}, SIAM J. Sci. Comput., 16 (1995), pp.~1269--1291.

\bibitem{Mi88}
{\sc W.~F. Mitchell}, {\em Unified Multilevel Adaptive Finite Element Methods
  for Elliptic Problems}, PhD thesis, Computer Science, University of Illinois
  at Urbana-Champaign, Urbana, IL, 1988.

\bibitem{On89}
{\sc M.~E.~G. Ong}, {\em Hierarchical basis preconditioners for second order
  elliptic problems in three dimensions}, PhD thesis, University of Washington,
  1989.

\bibitem{On94}
\leavevmode\vrule height 2pt depth -1.6pt width 23pt, {\em Uniform refinement
  of a tetrehedron}, SIAM J. Sci. Comput., 15 (1994), pp.~1134--1144.

\bibitem{Os90}
{\sc P.~Oswald}, {\em On function spaces related to finite element
  approximation theory}, Zeitschrift f{\"u}r Analysis und ihre Anwendungen, 9
  (1990), pp.~43--64.

\bibitem{Os94book}
\leavevmode\vrule height 2pt depth -1.6pt width 23pt, {\em Multilevel Finite
  Element Approximation Theory and Applications}, Teubner Skripten zur Numerik,
  B. G. Teubner, Stuttgart, 1994.

\bibitem{Se72}
{\sc E.~G. Sewell}, {\em Automatic generation of triangulations for piecewise
  polynomial approximation}, PhD thesis, Department of Mathematics, Purdue
  University, West Lafayette, IN, 1972.

\bibitem{St95}
{\sc R.~Stevenson}, {\em Robustness of the additive multiplicative frequency
  decomposition multi-level method}, Computing, 54 (1995), pp.~331--346.

\bibitem{St97}
\leavevmode\vrule height 2pt depth -1.6pt width 23pt, {\em A robust
  hierarchical basis preconditioner on general meshes}, Numer. Math., 78
  (1997), pp.~269--303.

\bibitem{StOs78}
{\sc E.~A. Storozhenko and P.~Oswald}, {\em Jackson's theorem in the spaces
  ${L_p(\Re^k),~0<p<1}$}, Siberian Math., 19 (1978), pp.~630--639.

\bibitem{VaWa2}
{\sc P.~S. Vassilevski and J.~Wang}, {\em Stabilizing the hierarchical basis by
  approximate wavelets, {I}: Theory}, Numer. Linear Alg. Appl., 4 Number 2
  (1997), pp.~103--126.

\bibitem{VaWa1}
\leavevmode\vrule height 2pt depth -1.6pt width 23pt, {\em Wavelet-like methods
  in the design of efficient multilevel preconditioners for elliptic {PDEs}},
  in Multiscale Wavelet Methods For Partial Differential Equations, W.~Dahmen,
  A.~Kurdila, and P.~Oswald, eds., Academic Press, 1997, ch.~1, pp.~59--105.

\bibitem{VaWa3}
\leavevmode\vrule height 2pt depth -1.6pt width 23pt, {\em Stabilizing the
  hierarchical basis by approximate wavelets, {II}: Implementation and
  numerical experiments}, SIAM J. Sci. Comput., 20 Number 2 (1998),
  pp.~490--514.

\bibitem{Ys86}
{\sc H.~Yserentant}, {\em On the multilevel splitting of finite element
  spaces}, Numer. Math., 49 (1986), pp.~379--412.

\bibitem{Zh88}
{\sc S.~Zhang}, {\em Multilevel iterative techniques}, PhD thesis, Pennsylvania
  State University, 1988.

\end{thebibliography}
